\documentclass[11pt,a4paper]{amsart}
\usepackage[T1]{fontenc}
\usepackage[utf8]{inputenc}
\usepackage{lmodern}
\usepackage{amsmath,amssymb,amsthm,mathtools}
\usepackage{enumitem}
\usepackage{microtype}
\usepackage{needspace}
\usepackage[a4paper,margin=30mm]{geometry}
\usepackage[hidelinks,unicode]{hyperref}
\usepackage{bookmark}
\newcommand{\doi}[1]{\href{https://doi.org/#1}{\nolinkurl{doi:#1}}}
\newtheorem{theorem}{Theorem}[section]
\newtheorem{proposition}[theorem]{Proposition}
\newtheorem{lemma}[theorem]{Lemma}
\newtheorem{corollary}[theorem]{Corollary}
\newtheorem{maintheorem}{Theorem}

\theoremstyle{definition}
\newtheorem{definition}[theorem]{Definition}
\newtheorem{question}[theorem]{Question}
\theoremstyle{remark}
\newtheorem{remark}[theorem]{Remark}
\DeclareMathOperator{\EL}{EL}
\DeclareMathOperator{\GL}{GL}
\DeclareMathOperator{\SL}{SL}

\newcommand{\F}{\mathbb F}
\newcommand{\Z}{\mathbb Z}
\newcommand{\N}{\mathbb N}
\newcommand{\Ctwo}{\Z/2\Z}
\newcommand{\BTV}[1]{{#1V}}
\hypersetup{pdftitle={Non-sofic Tarski monsters},
 pdfauthor={Kıvanç Ersoy},pdfsubject={Research manuscript}}
\title[Non-sofic Tarski monsters]{Non-sofic Tarski monsters}
\author{K\i van\c{c} Ersoy}
\address{Institut f\"ur Mathematik, Freie Universit\"at Berlin,
Arnimallee 7, 14195 Berlin, Germany}
\email{ersoy@zedat.fu-berlin.de}
\date{}
\begin{document}
\begin{abstract}
We prove the existence of non-sofic Tarski monsters of every sufficiently
large prime exponent and of torsion-free non-sofic Tarski monsters.
The proof combines a theorem of Sandhu--Thom with quotient theorems of
Semenov. These groups are minimal non-sofic, since all their proper
subgroups are cyclic. A theorem of Minasyan--Osin--Witzel gives
continuum many pairwise non-quasi-isometric torsion-free simple
examples. Kulikova's quotient theorem also yields non-simple minimal
non-sofic groups of exponent $p^m$ for every $m\geq2$ and every
sufficiently large prime $p$. We then study descending chains of
non-sofic subgroups in other constructions. We give proper
self-embeddings of the binary Leavitt unit group, higher
Brin--Thompson groups and the hyperbolic group studied by Douglas and
Mghirbi. The explicit Kun--Thom wreath products and doubles are locally
graded. These examples are $\omega$-non-sofic, and the locally graded
ones contain no minimal non-sofic subgroup.
\end{abstract}
\maketitle

\section{Introduction}
\label{sec:introduction}

For a prime $p$, a \emph{Tarski monster of exponent $p$} is an infinite
group whose non-trivial proper subgroups have order $p$. A
\emph{torsion-free Tarski monster} is a finitely generated non-cyclic
group whose non-trivial proper subgroups are infinite cyclic. We do
not include simplicity in the latter definition. For the classical
constructions, see Ol'shanskii \cite[Chapter 9, Sections 27--28]{OlshanskiiBook};
for the quotient constructions used here, see
\cite[Introduction]{SemenovFull} and
\cite[Corollary 4.10 and the preceding paragraph]{MinasyanOsinWitzel}.

\begin{maintheorem}\label{thm:main-existence}
For every sufficiently large prime $p$, there exists a non-sofic
Tarski monster of exponent $p$. There also exists a non-sofic
torsion-free Tarski monster.
\end{maintheorem}

Sandhu and Thom \cite[Corollary B]{SandhuThom} prove that some
non-trivial torsion-free hyperbolic group $H$ has no non-trivial sofic
quotient. Semenov's theorems \cite[Theorems A and B]{SemenovFull}
provide Tarski monsters as quotients of this $H$, and these quotients
must be non-sofic. The same $H$ works for all sufficiently large primes.

In \cite[Definitions 1.3 and 1.4]{Ersoy}, we studied \emph{minimal
non-sofic groups}, whose proper subgroups are all sofic, and
\emph{$\omega$-non-sofic groups}, which admit infinite strictly
descending chains of non-sofic subgroups. Since cyclic groups are
sofic \cite[Section II.3]{CapraroLupini}, Theorem~\ref{thm:main-existence}
gives minimal non-sofic groups. Their existence was not known when
\cite{Ersoy} was written.

In the torsion-free case, we obtain the following stronger result.

\begin{maintheorem}\label{thm:main-continuum}
There are $2^{\aleph_0}$ pairwise non-quasi-isometric torsion-free
simple non-sofic Tarski monsters. All of them are minimal non-sofic.
\end{maintheorem}

This follows from \cite[Corollary 4.10]{MinasyanOsinWitzel}: the groups
in that corollary are common quotients of all non-cyclic torsion-free
hyperbolic groups and have malnormal maximal subgroups.

Kulikova's quotient theorem \cite[Theorem 6]{KulikovaPreprint} gives
minimal non-sofic groups with non-trivial centre. The \emph{exponent}
of a group, when finite, is the least positive integer $n$ for which
$x^n=1$ holds for every element $x$.

\begin{maintheorem}\label{thm:main-primepowers}
There is an integer $p_0$ such that, for every prime $p>p_0$ and every
integer $m\geq1$, there exists a minimal non-sofic group $K_{p,m}$ of
exponent $p^m$ satisfying:
\begin{enumerate}[label=\textup{(\roman*)}]
\item every proper subgroup of $K_{p,m}$ is cyclic;
\item $Z(K_{p,m})$ is cyclic of order $p^{m-1}$;
\item $K_{p,m}/Z(K_{p,m})$ is a non-sofic Tarski monster of exponent $p$.
\end{enumerate}
\end{maintheorem}

For $m\geq2$, these groups are neither simple nor of prime exponent.
Taking $m=2$ answers \cite[Question 1.5]{Ersoy} negatively. Minimality
here concerns subgroups: $K_{p,m}$ has a proper non-sofic quotient when
$m\geq2$. Ivanov's embedding theorem also gives a two-generated
divisible periodic non-sofic group; see Corollary~\ref{cor:divisible}.

Section~\ref{sec:other-examples} asks which other constructions are
$\omega$-non-sofic. Proper self-embeddings give this property for the
binary Leavitt unit group \cite{OpenAI}, the non-sofic higher
Brin--Thompson groups \cite{SauerSchesler} and the hyperbolic group
studied by Douglas and Mghirbi \cite{DouglasMghirbi}. The explicit
Kun--Thom wreath products and doubles \cite{KunThom} are locally
graded, answering \cite[Question 1.9]{Ersoy} negatively. They are
$\omega$-non-sofic by \cite[Theorem 1.10]{Ersoy} and contain no minimal
non-sofic subgroup by \cite[Remark 1.11]{Ersoy}. We also discuss the
intermediate groups in Fournier-Facio's construction \cite{FournierFacio}
and the hyperbolic groups of Sandhu and Thom \cite{SandhuThom}.

\section*{AI Declaration}
ChatGPT (OpenAI) was used for mathematical discussions and assistance
with drafting and editing the manuscript. The author takes full
responsibility for the content of the paper.

\section{Preliminaries}
\label{sec:preliminaries}

All groups are countable and discrete. We write $H<G$ for a proper
subgroup and $\N=\{0,1,2,\ldots\}$. Permutations act on the left.

For $\sigma,\tau\in S_n$, the normalised Hamming metric is
\[
 d_{\mathrm H}(\sigma,\tau)
 =\frac{1}{n}\bigl|\{i:\sigma(i)\ne\tau(i)\}\bigr|.
\]
A group $G$ is \emph{sofic} if, for every finite $F\subseteq G$ and
$\varepsilon>0$, there are $n\geq1$ and a map $\theta:G\to S_n$ with
$\theta(1)=1$ such that
\[
 d_{\mathrm H}(\theta(gh),\theta(g)\theta(h))<\varepsilon
 \quad(g,h\in F)
\]
and
\[
 d_{\mathrm H}(\theta(g),1)>1-\varepsilon
 \quad(g\in F\setminus\{1\}).
\]
We use the equivalent formulation in
\cite[Definition II.1.2 and Exercise II.1.8]{CapraroLupini}.

Amenable groups are sofic
\cite[Proposition II.3.1]{CapraroLupini}, as are residually finite
groups \cite[Section II.3]{CapraroLupini}. Soficity passes to subgroups
\cite[Proposition II.4.1(1)]{CapraroLupini} and is local: a group is
sofic if all its finitely generated subgroups are sofic
\cite[Exercise II.3.2]{CapraroLupini}. Free products of sofic groups
are sofic \cite[Theorem 1(2)]{ElekSzabo2006}. Finally, if
$N\trianglelefteq G$ is sofic and $G/N$ is amenable, then $G$ is sofic
\cite[Proposition II.4.1(5)]{CapraroLupini}; see also
\cite[Proposition 4.5]{ElekSzabo2004}. These facts are recalled in
\cite[Proposition 2.1]{Ersoy}.

\begin{definition}[{\cite[Definitions 1.3 and 1.4]{Ersoy}}]
A group is \emph{minimal non-sofic} if it is non-sofic and every proper
subgroup is sofic. A group $G$ is \emph{$\omega$-non-sofic} if it has a
chain
\[
 G=G_0>G_1>G_2>\cdots
\]
in which every $G_i$ is non-sofic.
\end{definition}

Recall from \cite[Section 1]{Ersoy} that a group is \emph{locally
graded} if every non-trivial finitely generated subgroup has a proper
subgroup of finite index. A \emph{proper self-embedding} is an injective
endomorphism with proper image; a group is \emph{co-Hopfian} if it has
no such endomorphism.

\begin{lemma}\label{lem:finite-index}
If $G$ is non-sofic, then every subgroup of finite index in $G$ is
non-sofic. More generally, if $N\trianglelefteq G$ and $G/N$ is amenable,
then $N$ is non-sofic.
\end{lemma}
\begin{proof}
The second assertion is the contrapositive of
\cite[Proposition II.4.1(5)]{CapraroLupini}. For the first, let $L\leq G$
have finite index and let $N=\bigcap_{g\in G}gLg^{-1}$ be its core.
The action on the finite coset space $G/L$ has kernel $N$, so
$N\trianglelefteq G$, $N\leq L$ and $G/N$ is finite. If $L$ were sofic,
then $N$ would be sofic by
\cite[Proposition II.4.1(1)]{CapraroLupini}, contradicting the second
assertion. Compare the finite-index argument in
\cite[proof of Theorem 1.10]{Ersoy}.
\end{proof}

The following lemma collects
\cite[Proposition 2.4(1)--(2) and Remark 1.11]{Ersoy} and elementary
consequences needed below.

\begin{lemma}\label{lem:minimal-restrictions}
A minimal non-sofic group $M$ is finitely generated, has no proper
subgroup of finite index and has no non-trivial amenable quotient. It
is co-Hopfian and is neither locally graded nor $\omega$-non-sofic.
\end{lemma}
\begin{proof}
If $M$ were not finitely generated, all its finitely generated subgroups
would be proper and sofic. Locality
\cite[Exercise II.3.2]{CapraroLupini} would imply that $M$ is sofic.
A proper subgroup of finite index is excluded by
Lemma~\ref{lem:finite-index} and minimality. If $M/N$ were a non-trivial
amenable quotient, then $N<M$ would be sofic and
\cite[Proposition II.4.1(5)]{CapraroLupini} would give a contradiction.

An injective endomorphism with proper image would give a proper
non-sofic subgroup isomorphic to $M$. Thus $M$ is co-Hopfian. Since $M$
is non-trivial and finitely generated, local gradedness would give a
proper subgroup of finite index. Finally, an $\omega$-non-sofic chain
would already provide a proper non-sofic subgroup at its first step.
\end{proof}

We will also use two elementary extension observations.

\begin{lemma}\label{lem:graded}
The class of locally graded groups is closed under subgroups and
extensions. Locally finite groups and residually finite groups are
locally graded. Consequently, a locally graded group contains no minimal
non-sofic subgroup.
\end{lemma}
\begin{proof}
Inheritance by subgroups follows immediately from the definition.
Suppose $N\trianglelefteq L$ and both $N$ and $L/N$ are locally graded.
Let $1\ne A\leq L$ be finitely generated, and let $\pi:L\to L/N$ be the
quotient map. If $A\leq N$, use local gradedness of $N$. Otherwise,
$\pi(A)$ is a non-trivial finitely generated group. A proper subgroup
of finite index in $\pi(A)$ pulls back to a proper subgroup of finite
index in $A$.

A non-trivial finitely generated subgroup of a locally finite group is
finite. In a residually finite group, a non-trivial element of such a
subgroup survives in a finite quotient; the kernel of the restricted
map has proper finite index. The last assertion now follows from
Lemma~\ref{lem:minimal-restrictions}.
\end{proof}

The following special case will be applied to the Kun--Thom wreath products.

\begin{lemma}\label{lem:lf-by-rf}
Let $L$ be a group and let $N\trianglelefteq L$. If $N$ is locally finite
and $L/N$ is residually finite, then $L$ is locally graded.
\end{lemma}
\begin{proof}
Let $1\ne A\leq L$ be finitely generated. If $A\leq N$, then $A$ is
finite, so $\{1\}$ has proper finite index in $A$. Otherwise, residual
finiteness of $L/N$ gives a homomorphism from $L/N$ to a finite group
which is non-trivial on the image of $A$. Its composite with the
quotient map restricts to a non-trivial homomorphism on $A$, whose
kernel has proper finite index.
\end{proof}

\Needspace{8\baselineskip}
The following observation also underlies the conditional discussion of
Thompson groups in \cite[Section 1]{Ersoy}.

\begin{lemma}\label{lem:self-embedding}
A non-sofic group with a proper self-embedding is $\omega$-non-sofic.
\end{lemma}
\begin{proof}
Let $\varphi:G\to G$ be injective with $\varphi(G)<G$. Applying
$\varphi^i$ to this strict inclusion gives
$\varphi^{i+1}(G)<\varphi^i(G)$. Every group in the resulting chain is
isomorphic to $G$ and hence non-sofic.
\end{proof}

The property of being $\omega$-non-sofic passes to overgroups. Indeed,
if $L<G$ has such a chain, prepend $G$ to that chain; $G$ is non-sofic
because it contains $L$ \cite[Proposition II.4.1(1)]{CapraroLupini}.

The finite-index argument can also be applied to inverse images of
subgroups in a residually finite quotient. This is a variant of the
chain argument in \cite[proof of Theorem 1.10]{Ersoy}; the intersection
is recorded here for later use.

\begin{proposition}\label{prop:rf-chain}
Suppose that $L$ is non-sofic and $\pi:L\twoheadrightarrow Q$, where $Q$
is infinite and residually finite. There is a chain
\[
 L=L_0>L_1>L_2>\cdots
\]
of finite-index normal non-sofic subgroups such that
\[
 \bigcap_{i\geq0}L_i=\ker\pi.
\]
In particular, $L$ is $\omega$-non-sofic.
\end{proposition}
\begin{proof}
Enumerate $Q\setminus\{1\}=\{q_1,q_2,\ldots\}$. Choose finite-index
normal subgroups $N_j\trianglelefteq Q$ with $q_j\notin N_j$, and put
$H_0=Q$ and $H_j=N_1\cap\cdots\cap N_j$. Their intersection is trivial.
There are infinitely many distinct terms: otherwise the chain would
stabilize at the trivial subgroup, which has infinite index in $Q$.
Deleting repetitions gives a strictly descending chain
$Q=Q_0>Q_1>\cdots$ of finite-index normal subgroups with trivial
intersection. Set $L_i=\pi^{-1}(Q_i)$. Surjectivity makes all inclusions
strict, and the asserted intersection follows from inverse images.
Each $L_i$ is non-sofic by Lemma~\ref{lem:finite-index}.
\end{proof}

\section{Non-sofic Tarski monsters}
\label{sec:existence}

\subsection{The hyperbolic input and Semenov's quotients}

Sandhu and Thom prove the following result.

\begin{theorem}[Sandhu--Thom, {\cite[Corollary B]{SandhuThom}}]
\label{thm:ST}
There is a non-trivial torsion-free word-hyperbolic group $H$ such that
\begin{equation}\label{eq:no-sofic-quotient}
 N\trianglelefteq H,\quad N\ne H
 \quad\Longrightarrow\quad H/N\text{ is non-sofic}.
\end{equation}
\end{theorem}

Fix such an $H$. It is non-sofic by taking $N=1$ in
\eqref{eq:no-sofic-quotient}, and hence non-elementary: virtually cyclic
groups are amenable and therefore sofic
\cite[Sections II.3--II.4]{CapraroLupini}. We need the following
quotient theorems of Semenov. They were announced in
\cite{Semenov1993}; we use the statements and numbering of the full
manuscript \cite{SemenovFull}.

\begin{theorem}[Semenov, {\cite[Theorems A and B]{SemenovFull}}]
\label{thm:Semenov}
Let $\Gamma$ be a non-elementary torsion-free hyperbolic group.
\begin{enumerate}[label=\textup{(\roman*)}]
\item There is $n_0=n_0(\Gamma)$ such that, for every odd integer
$n\geq n_0$, some infinite quotient of $\Gamma$ has all its proper
subgroups cyclic of order dividing $n$.
\item There is an infinite quotient of $\Gamma$ whose non-trivial
proper subgroups are infinite cyclic.
\end{enumerate}
\end{theorem}

\begin{proof}[Proof of Theorem~\ref{thm:main-existence}]
Let $p\geq n_0(H)$ be an odd prime. By
Theorem~\ref{thm:Semenov}\textup{(i)}, the group $H$ has an infinite
quotient $T_p$ whose proper subgroups are cyclic of order dividing $p$.
By \eqref{eq:no-sofic-quotient}, $T_p$ is non-sofic, so it is not cyclic.
Every non-identity element therefore generates a proper subgroup of
order $p$. Thus $T_p$ is a Tarski monster of exponent $p$.

By Theorem~\ref{thm:Semenov}\textup{(ii)}, the same group $H$ has an
infinite quotient $T$ whose non-trivial proper subgroups are infinite
cyclic. Again \eqref{eq:no-sofic-quotient} makes $T$ non-sofic and hence
non-cyclic. Its cyclic subgroups are proper, so every non-identity
element has infinite order. Finally, $T$ is finitely generated as a
quotient of the hyperbolic group $H$. It is therefore a torsion-free
Tarski monster.
\end{proof}

\begin{corollary}\label{cor:monsters-minimal}
Every non-sofic Tarski monster is minimal non-sofic. Such a group is
co-Hopfian and is neither locally graded nor $\omega$-non-sofic.
\end{corollary}
\begin{proof}
All its proper subgroups are cyclic, hence sofic
\cite[Section II.3]{CapraroLupini}. Apply
Lemma~\ref{lem:minimal-restrictions}.
\end{proof}

Here the prime threshold depends on the fixed group $H$; the monsters
of different prime exponents are separate quotients of $H$.

\begin{lemma}\label{lem:tarski-basic}
Every Tarski monster of prime exponent is simple and two-generated.
Every torsion-free minimal non-sofic group with cyclic proper subgroups
is two-generated.
\end{lemma}
\begin{proof}
Let $T$ be a Tarski monster of exponent $p$. Choose $1\ne a\in T$ and
$b\notin\langle a\rangle$. The subgroup $\langle a,b\rangle$ cannot
be proper, since a proper non-trivial subgroup has order $p$.
Thus $T=\langle a,b\rangle$.

If $1<N\trianglelefteq T$ were proper, then $|N|=p$. Choose
$x\notin N$. The subgroup $N\langle x\rangle$ would have order $p^2$,
contrary to the description of proper subgroups of the infinite group
$T$. Thus $T$ is simple.

A minimal non-sofic group is not abelian, since abelian groups are
amenable and hence sofic
\cite[Section II.3 and Proposition II.3.1]{CapraroLupini}. In the torsion-free case under consideration, two
non-commuting elements generate a non-cyclic subgroup, which cannot be
proper. This proves the last assertion.
\end{proof}

\begin{remark}\label{rem:odd-bounds}
Semenov's Theorem A applies to every sufficiently large odd integer
$n$, not just to primes. Therefore the same argument also gives minimal
non-sofic quotients of $H$ all of whose proper subgroups are cyclic of
order dividing $n$. The exact prime-power exponents established below
will be obtained from Kulikova's stronger quotient theorem.
\end{remark}

\subsection{Continuum many torsion-free simple examples}

\begin{proof}[Proof of Theorem~\ref{thm:main-continuum}]
By \cite[Corollary 4.10]{MinasyanOsinWitzel}, there is a family of
$2^{\aleph_0}$ pairwise non-quasi-isometric torsion-free Tarski monsters,
each a quotient of every non-cyclic torsion-free hyperbolic group.
Moreover, all maximal subgroups of each member are malnormal. In
particular, each member $T$ is a non-trivial quotient of our fixed
$H$. Equation~\eqref{eq:no-sofic-quotient} makes $T$ non-sofic, and its
cyclic proper subgroups make it minimal non-sofic.

The malnormality assertion implies simplicity.
Suppose $1<N\trianglelefteq T$ is proper. Since $T$ is finitely generated,
$N$ is contained in a maximal proper subgroup $M$: the union of a chain
of proper subgroups containing $N$ is still proper, since otherwise
one member would contain a finite generating set of $T$, and Zorn's
lemma applies. Compare the maximal-subgroup argument in
\cite[proof of Proposition 2.4(1)]{Ersoy}. Choose $g\notin M$.
Normality gives
\[
 N=gNg^{-1}\leq M\cap gMg^{-1}=1,
\]
contradicting malnormality. Thus $T$ is simple. The groups are pairwise non-quasi-isometric by the cited corollary.
\end{proof}

\Needspace{10\baselineskip}
\section{Further consequences}
\label{sec:further-existence}

We keep the group $H$ from Theorem~\ref{thm:ST}. The following
applications use other quotient and embedding theorems.

\subsection{Prime-power exponents and non-simple examples}

We use \cite[Theorem 6]{KulikovaPreprint}, whose published version is
\cite{Kulikova}.

\begin{theorem}[Kulikova, {\cite[Theorem 6]{KulikovaPreprint}}]\label{thm:Kulikova}
Let $\Gamma$ be a non-cyclic torsion-free hyperbolic group. For every
sufficiently large prime $p$ and every integer $k\geq0$, $\Gamma$ has
a non-abelian quotient $K$ of exponent $p^{k+1}$ with a central cyclic
subgroup $D$ of order $p^k$. Every proper subgroup of $K$ is cyclic,
and every subgroup of $K$ either contains $D$ or is contained in $D$.
\end{theorem}

\begin{proof}[Proof of Theorem~\ref{thm:main-primepowers}]
Apply Theorem~\ref{thm:Kulikova} to $H$ with $k=m-1$. Write $K=K_{p,m}$
for the resulting quotient and let $D\leq Z(K)$ be the cyclic subgroup
of order $p^{m-1}$ supplied by that theorem. The group $K$ is non-trivial,
so \eqref{eq:no-sofic-quotient} makes it non-sofic. Since its proper
subgroups are cyclic, $K$ is minimal non-sofic. In particular, it is
infinite.

The quotient $K/D$ is infinite and is also a quotient of $H$, so it is
non-sofic. Let $D<U<K$. By Theorem~\ref{thm:Kulikova}, the group $U$ is
cyclic. Its order is a power of $p$ which is larger than $p^{m-1}$ and
at most $p^m$, and hence is exactly $p^m$. It follows that every
non-trivial proper subgroup of $K/D$ has order $p$. Thus $K/D$ is a
Tarski monster of exponent $p$ and is simple by
Lemma~\ref{lem:tarski-basic}.

Since this infinite simple group has trivial centre, the image of
$Z(K)$ in $K/D$ is trivial. Therefore $Z(K)\leq D$, and the reverse
inclusion is part of the construction. Consequently $Z(K)=D$. For
$m\geq2$, this is a non-trivial proper normal subgroup of $K$, while
$p^m$ is not prime.
\end{proof}

\begin{corollary}\label{cor:not-prime}
A periodic minimal non-sofic group need not have prime exponent, and
a minimal non-sofic group need not be simple.
\end{corollary}
\begin{proof}
Take $m=2$ in Theorem~\ref{thm:main-primepowers}.
\end{proof}

The first assertion answers \cite[Question 1.5]{Ersoy} negatively.
These examples do not answer \cite[Question 1.6]{Ersoy}, which asks
whether every periodic minimal non-sofic group has bounded exponent:
each individual group constructed in
Theorem~\ref{thm:main-primepowers} still has finite exponent.

\subsection{A consequence for periodic groups}

The prime-exponent examples make
\cite[Proposition 1.7 and Corollary 1.8]{Ersoy} unconditional.

\begin{corollary}\label{cor:divisible}
There exists a two-generated divisible periodic non-sofic group of
unbounded exponent.
\end{corollary}
\begin{proof}
Choose a sufficiently large odd prime $p$ and a non-sofic Tarski
monster $T_p$ as in Theorem~\ref{thm:main-existence}. It is countable
and has no involutions. Ivanov's embedding theorem
\cite[Theorem 20, p.~288]{IvanovOlshanskii} embeds $T_p$ in a
two-generated divisible periodic group $D$. Since $D$ contains the
non-sofic subgroup $T_p$, it is non-sofic.

If $D$ had finite exponent $e$, divisibility would give, for every
$g\in D$, an element $h\in D$ with $h^e=g$. This would force $g=1$,
a contradiction. Thus $D$ has unbounded exponent. In particular,
$T_p<D$, so this $D$ is not minimal non-sofic.
\end{proof}

\section{Subgroup structure of other non-sofic groups}
\label{sec:other-examples}

In \cite[Section 1 and Definition 1.4]{Ersoy}, we studied
$\omega$-non-sofic groups, that is, groups with infinite strictly descending
chains of non-sofic subgroups. This leads us to ask which of the newly
constructed non-sofic groups have this property. We prove
$\omega$-non-soficity for several of the families below, using proper
self-embeddings or local gradedness and \cite[Theorem 1.10]{Ersoy}.
We also indicate the cases left open. The non-soficity results used
in this section are taken from the cited papers.

\subsection{The binary Leavitt unit group}
\label{sec:leavitt}

Let $R=L_{\F_2}(1,2)$, with presentation
\begin{equation}\label{eq:leavitt-relations}
 R=\F_2\langle s_0,s_1,t_0,t_1\mid
 t_i s_j=\delta_{ij}1,\quad s_0t_0+s_1t_1=1\rangle;
\end{equation}
see \cite[Definition 1.1.3]{AbramsAraSilesMolina}. By
\cite[Proposition 1.3.2]{AbramsAraSilesMolina}, this is the Leavitt
path algebra of the graph with one vertex and two loops. The $s_i$
are the loops and the $t_i$ their ghost edges; the vertex is the
identity, and \eqref{eq:leavitt-relations} gives (CK1)--(CK2)
\cite[Definition 1.2.3 and Lemma 1.2.12(iv)]{AbramsAraSilesMolina}.
The vertex is non-zero by
\cite[Corollary 1.5.13]{AbramsAraSilesMolina}.

The group $R^\times$ is non-sofic
\cite[Chapter 3, Theorem 1.1]{OpenAI}. Moreover,
$R^\times\cong\EL_9(R)$ by \cite[Corollary 4.4]{KhanhThanh}.
We give proper self-embeddings in both descriptions.

Put $e=s_0t_0$ and $f=s_1t_1$. The defining relations give
\begin{equation}\label{eq:idempotents}
 e+f=1,\qquad e^2=e,\qquad f^2=f,\qquad ef=fe=0.
\end{equation}
Both idempotents are non-zero, since $t_0es_0=t_1fs_1=1$; see also
\cite[Remark 1.3.3 and Example 3.1.3]{AbramsAraSilesMolina}.
The corner algebra $eRe$ has identity $e$, not $1$. The corresponding
module decomposition is $R=Re\oplus Rf$, with $Re\cong R$
\cite[Proposition 3.0.1 and Example 3.1.3]{AbramsAraSilesMolina}.
The ring isomorphism needed below is explicit.

\begin{lemma}\label{lem:leavitt-corner}
The map $c:R\to eRe$ given by $c(r)=s_0rt_0$ is a ring isomorphism
with inverse $a\mapsto t_0as_0$. It sends $1$ to $e$.
\end{lemma}
\begin{proof}
The map is additive, its image lies in $eRe$, and
\[
 c(r)c(s)=s_0r(t_0s_0)st_0=c(rs).
\]
For $r\in R$ and $a\in eRe$, we have
$t_0c(r)s_0=r$ and $c(t_0as_0)=eae=a$.
\end{proof}

If $a,b\in eRe$ satisfy $ab=ba=e$, then
\[
 (a+f)(b+f)=(b+f)(a+f)=1.
\]
Also $(a+f)(a'+f)=aa'+f$ for $a,a'\in(eRe)^\times$. Hence
$a\mapsto a+f$ embeds $(eRe)^\times$ in $R^\times$, as in
\cite[Chapter 3, Section 3.1]{OpenAI}. Composing with $c$ gives the
following map.

\begin{proposition}\label{prop:unit-embedding}
The map
\[
 \varphi:R^\times\longrightarrow R^\times,\qquad
 \varphi(u)=s_0ut_0+f
\]
is a proper self-embedding. Consequently $R^\times$ is
$\omega$-non-sofic.
\end{proposition}
\begin{proof}
The preceding identities show that $\varphi$ is a homomorphism, with
$\varphi(u)^{-1}=s_0u^{-1}t_0+f$. It is injective since
$t_0\varphi(u)s_0=u$.

Put $w=1+s_1t_0$. Since $(s_1t_0)^2=0$ and $R$ has characteristic
two, $w^2=1$. Every element in the image satisfies
$f\varphi(u)e=0$, whereas
\[
 fwe=s_1t_0\ne0,
\]
as $t_1(s_1t_0)s_0=1$. Thus $w\notin\varphi(R^\times)$.
Now apply Lemma~\ref{lem:self-embedding}.
\end{proof}

For comparison with the matrix-ring description in
\cite[Section 7.2]{AbramsAraSilesMolina} and
\cite[Proposition 2.4]{KhanhThanh}, the maps
\[
 r\longmapsto(t_i r s_j)_{i,j=0,1},\qquad
 (a_{ij})\longmapsto\sum_{i,j=0}^1s_i a_{ij}t_j
\]
are inverse ring isomorphisms $R\cong M_2(R)$, by
\eqref{eq:leavitt-relations}. Under this isomorphism,
\[
 \varphi(u)\longmapsto\begin{pmatrix}u&0\\0&1\end{pmatrix},
 \qquad w\longmapsto\begin{pmatrix}1&0\\1&1\end{pmatrix}.
\]
Thus $\varphi$ is a group homomorphism on units, not a ring
homomorphism on $R$: the same formula sends $0$ to $f$.

\Needspace{8\baselineskip}
Write $\EL_9(R)$ for the subgroup of $\GL_9(R)$ generated by
$I_9+rE_{ij}$ with $r\in R$ and $i\ne j$
\cite[Section 3]{KhanhThanh}.

\begin{proposition}\label{prop:matrix-embedding}
The map
\[
 \Phi:\EL_9(R)\longrightarrow\EL_9(R),\qquad
 \Phi((a_{ij}))=(c(a_{ij})+\delta_{ij}f)_{i,j=1}^9
\]
is a proper self-embedding. Consequently $\EL_9(R)$ is
$\omega$-non-sofic.
\end{proposition}
\begin{proof}
Since $c$ is additive and multiplicative and $c(r)f=fc(r)=0$,
\[
 (\Phi(A)\Phi(B))_{ij}=c((AB)_{ij})+\delta_{ij}f.
\]
Also $\Phi(I_9)=I_9$ and
$\Phi(I_9+rE_{ij})=I_9+c(r)E_{ij}$ for $i\ne j$.
Thus $\Phi$ is a homomorphism of the elementary group. It is
injective because $t_0\Phi(A)_{ij}s_0=a_{ij}$.

Every off-diagonal entry in its image belongs to $eRe$. But
$s_1\notin eRe$, since $es_1=0$ and $s_1\ne0$. Hence
$I_9+s_1E_{12}$ is not in the image. Non-soficity follows from
$\EL_9(R)\cong R^\times$, and Lemma~\ref{lem:self-embedding} applies.
\end{proof}

\subsection{Fournier-Facio's torsion-free construction}
\label{sec:ff}

Fournier-Facio constructs a finitely presented torsion-free non-sofic
group $Q$ with property~$(T)$
\cite[Theorem 1.3 and Section 2]{FournierFacio}.
His construction starts with a finitely presented torsion-free group
$U$ which contains every finitely presented torsion-free group.
It embeds $U$ in a finitely presented torsion-free group $P$ with
property~$(T)$. We first record a consequence for these two groups.

\begin{proposition}\label{prop:universal-tf}
Every finitely presented torsion-free group $U$ which contains all
finitely presented torsion-free groups is non-sofic and has a proper
self-embedding. In particular, the groups $U$ and $P$ above are
$\omega$-non-sofic and are not minimal non-sofic.
\end{proposition}
\begin{proof}
By \cite[Theorem 1.3]{FournierFacio}, there is a finitely presented
torsion-free non-sofic group. Universality embeds it in $U$, so $U$
is non-sofic by \cite[Proposition II.4.1(1)]{CapraroLupini}.

The group $U\times U$ is torsion-free and finitely presented: finite
presentations of the two factors, together with the finitely many
commutators between their generators, give a finite presentation.
Universality provides an embedding $j:U\times U\hookrightarrow U$.
The map $u\mapsto j(u,1)$ is injective and has proper image. Indeed,
for $v\ne1$, injectivity of $j$ gives $j(1,v)\notin j(U\times\{1\})$.
Apply Lemma~\ref{lem:self-embedding}.
The same argument applies to $P$, which is finitely presented,
torsion-free and contains $U$, and is therefore universal as well.
\end{proof}

For the final quotient $Q$, we use the notation of
\cite[Section 2]{FournierFacio}. There are subgroups
$P_1\times P_2\times S\leq P$, with $P_i\cong P$ and $S$ infinite
simple. Let $E$ be the double HNN extension with stable letters
$u_1,u_2$ satisfying $u_iPu_i^{-1}=P_i$.
The last step takes a quotient $\pi:E\twoheadrightarrow Q$ which is
injective on $S$. Put
\[
 \Gamma=\pi(P),\qquad S_0=\pi(S),\qquad t_i=\pi(u_i).
\]
Then $S_0\leq\Gamma$, $S_0\cong S$, and
\[
 t_1\Gamma t_1^{-1}=\pi(P_1)\leq\Gamma,
 \qquad [\pi(P_1),S_0]=1.
\]

\begin{proposition}\label{prop:ff-compression}
In this construction, $\Gamma<Q$ and conjugation by $t_1$ is a proper
self-embedding of $\Gamma$. If $\Gamma$ is non-sofic, then $Q$ is
$\omega$-non-sofic.
\end{proposition}
\begin{proof}
If $\pi(P_1)=\Gamma$, then $S_0\leq Z(\Gamma)$, which is impossible
because $S_0$ is non-abelian. Thus $t_1\Gamma t_1^{-1}<\Gamma$.
Conjugation is injective, so it is a proper self-embedding of $\Gamma$.
Also $t_1\notin\Gamma$, since conjugation by an element of $\Gamma$
maps $\Gamma$ onto itself. Hence $\Gamma<Q$.
If $\Gamma$ is non-sofic, the chain
\[
 Q>\Gamma>t_1\Gamma t_1^{-1}>t_1^2\Gamma t_1^{-2}>\cdots
\]
has the required property.
\end{proof}

The construction requires $\pi$ to be injective on $S$, but not on $P$.
Thus the non-soficity of $P$ does not establish the hypothesis on
$\Gamma$ in Proposition~\ref{prop:ff-compression}.

\subsection{The Kun--Thom wreath products}
\label{sec:kt}

Following \cite[Section 1]{KunThom}, a subgroup $\Gamma\leq G$ is
\emph{infranormal} if the set
\[
 P_\Gamma=\{g\in G:g\Gamma g^{-1}\leq\Gamma\}
\]
generates $G$. Kun and Thom prove that
\[
 \left(\bigoplus_{G/\Gamma}\Ctwo\right)\rtimes G
 \quad\hbox{and}\quad G*_{\Gamma}G
\]
are non-sofic whenever $\Gamma$ is infranormal but not normal in $G$
and both $\Gamma$ and $G$ have property~$(T)$
\cite[Theorem A]{KunThom}.

For the subgroup arguments, we use their explicit residually finite
pairs \cite[Theorem E]{KunThom}. Fix a prime power $q$ and integers
$r,d\geq3$, and put
\[
 R_+=\F_q[x_1,\ldots,x_d],\qquad
 R=\F_q[x_1^{\pm1},\ldots,x_d^{\pm1}],
\]
\begin{equation}\label{eq:kt-pair}
 \Gamma=\EL_r(R_+),\qquad G=\EL_r(R)\rtimes\SL_d(\Z),
\end{equation}
where $\SL_d(\Z)$ acts by monomial substitutions. The cited theorem
states that both groups are residually finite and have property~$(T)$,
and that $\Gamma$ is infranormal but not normal in $G$.
Let $B=\bigoplus_{G/\Gamma}\Ctwo$ and $\Lambda=B\rtimes G$.

\begin{proposition}\label{prop:kt-locally-graded}
For the pair in \eqref{eq:kt-pair}, $\Lambda$ is locally graded and
non-sofic. It is therefore $\omega$-non-sofic and contains no minimal
non-sofic subgroup.
\end{proposition}
\begin{proof}
A finitely generated subgroup of $B$ is supported on finitely many
coordinates, and is therefore finite. Thus $B$ is locally finite.
The quotient $\Lambda/B\cong G$ is residually finite by
\cite[Theorem E]{KunThom}. Hence Lemma~\ref{lem:lf-by-rf} shows that
$\Lambda$ is locally graded. It is non-sofic by
\cite[Theorems A and E]{KunThom}. Now
\cite[Theorem 1.10]{Ersoy} gives $\omega$-non-soficity. Since local
gradedness passes to subgroups and minimal non-sofic groups are not
locally graded \cite[Remark 1.11]{Ersoy}, $\Lambda$ contains no minimal
non-sofic subgroup.
\end{proof}

In \cite[Question 1.9]{Ersoy}, we asked whether every locally graded
group is sofic. Proposition~\ref{prop:kt-locally-graded} answers this
question negatively: $\Lambda$ is locally graded but non-sofic.
The residual finiteness used here belongs to the explicit pairs in
\cite[Theorem E]{KunThom}; it is not a hypothesis of their general
Theorem A.

\subsection{The Kun--Thom amalgamated doubles}\label{sec:doubles}

Use the pair $(\Gamma,G)$ in \eqref{eq:kt-pair}, and put
$\Delta=G*_{\Gamma}G$, with the same inclusion in both vertex groups.
We first describe the kernel of the folding retraction. The argument
uses the tree of an amalgam from
\cite[Chapter I, Section 4.1, Theorem 7]{Serre} and the free-action
criterion \cite[Chapter I, Section 3.3, Theorem 4]{Serre}.
The rank calculation below is an application of the free-basis
description in \cite[Chapter I, Section 3.3, Theorem $4'$]{Serre}.

\begin{lemma}\label{lem:folding}
Let $\Gamma<G$ be countable, let $\Delta=G*_{\Gamma}G$, and let
$\rho:\Delta\to G$ restrict to the identity on both vertex groups.
Then $K=\ker\rho$ is free. If $[G:\Gamma]$ is infinite, $K$ is a free
group of countably infinite rank.
\end{lemma}
\begin{proof}
The restrictions of $\rho$ to both copies of $G$ are injective. Since
$K$ is normal, its intersection with every conjugate of either vertex
group is trivial. Thus $K$ acts freely on the Bass--Serre tree and is
free by \cite[Chapter I, Section 3.3, Theorem 4]{Serre}.

More explicitly, the quotient graph of this tree by $K$ has two
vertices, one of each type, because $\rho$ maps each vertex group
onto $G$. Its edges are indexed by $G/\Gamma$, and every edge joins
the two vertices. Indeed, vertex orbits are the double cosets
$K\backslash\Delta/G_i$ for the two vertex groups $G_i$, and edge
orbits are $K\backslash\Delta/\Gamma$; applying $\rho$ identifies
these sets with singletons and $G/\Gamma$, respectively. Choosing
one edge as a spanning tree gives a free basis indexed by all the
remaining edges, as in
\cite[Chapter I, Section 3.3, Theorem $4'$]{Serre}. Infinite index therefore
gives countably infinite rank.
\end{proof}

\begin{theorem}\label{thm:double-chain}
For the pair $(\Gamma,G)$ in \eqref{eq:kt-pair}, the double
$\Delta=G*_{\Gamma}G$ is locally graded and non-sofic. It has a strictly
descending chain
\[
 \Delta=\Delta_0>\Delta_1>\Delta_2>\cdots
\]
of finite-index normal locally graded non-sofic subgroups whose
intersection is the kernel of the folding retraction. This kernel is
a free group of countably infinite rank. In particular, $\Delta$ is
$\omega$-non-sofic and is not minimal non-sofic.
\end{theorem}
\begin{proof}
Non-soficity is \cite[Theorem A]{KunThom}. Since
$\Gamma\leq\EL_r(R)$ and $G/\EL_r(R)\cong\SL_d(\Z)$ is infinite,
we have $[G:\Gamma]=\infty$.

By Lemma~\ref{lem:folding}, the folding kernel $K$ is free of countably
infinite rank. Free groups are residually finite
\cite[Section II.3, p.~24]{CapraroLupini} and thus locally graded by
Lemma~\ref{lem:graded}. The exact sequence
\[
 1\longrightarrow K\longrightarrow\Delta
 \mathrel{\mathop{\longrightarrow}\limits^{\rho}}G
 \longrightarrow1
\]
and Lemma~\ref{lem:graded} show that $\Delta$ is locally graded.
Apply Proposition~\ref{prop:rf-chain} to $\rho$. This gives the chain
with intersection $K$, and its terms are locally graded by inheritance.
\end{proof}

\begin{corollary}\label{cor:no-minimal-subgroup}
There exist non-sofic groups containing no minimal non-sofic subgroup.
In particular, neither the wreath product $\Lambda$ in
Proposition~\ref{prop:kt-locally-graded} nor any of the groups $\Delta_i$
in Theorem~\ref{thm:double-chain} contains a minimal non-sofic subgroup.
\end{corollary}
\begin{proof}
All these groups are locally graded. Apply Lemma~\ref{lem:graded},
whose last assertion is based on \cite[Remark 1.11]{Ersoy}.
\end{proof}

The intersection in Theorem~\ref{thm:double-chain} is free and hence
sofic \cite[Section II.3]{CapraroLupini}. Thus the doubles also give strictly descending chains of
non-sofic groups with sofic intersection.

\subsection{Higher Brin--Thompson groups}
\label{sec:brin}

Write $\mathcal C=\{0,1\}^{\N}$ for the Cantor space and $\BTV{n}$ for
the $n$-dimensional Brin--Thompson group acting on $\mathcal C^n$.
Here $n$ is a dimension: this notation is different from the
$n$-ary Higman--Thompson notation $V_n$. An element of $\BTV{n}$ is
given by finitely many coordinatewise prefix replacements between
dyadic rectangles forming partitions of $\mathcal C^n$. These groups
were introduced by Brin \cite{Brin2004}; for the prefix-replacement model
used here, see \cite[Section 5.1]{SauerSchesler}.

The classical Thompson groups $F,T,V$ were possible non-sofic
examples in \cite[Section 1]{Ersoy}. They are now known to be sofic
\cite[Theorem 1.1 and Corollary 1.2]{ChaudkhariJuschenkoMuliarchyk}.
In particular $1V=V$ is sofic.

For sufficiently large $n$, however, $nV$ has property~$(T)$
\cite[Theorem 1.1]{SauerSchesler}. Fournier-Facio proves in the appendix
to that paper that these groups are non-hyperlinear, hence non-sofic
\cite[Theorem A.1]{SauerSchesler}; see
\cite[Proposition II.2.5]{CapraroLupini} for the implication from
soficity to hyperlinearity. We use the standard localization embedding
from \cite[Notation A.6]{SauerSchesler}.

\begin{proposition}\label{prop:brin-embedding}
For every integer $n\geq1$, the group $\BTV{n}$ admits a proper
self-embedding. Consequently, for all sufficiently large $n$,
$\BTV{n}$ is $\omega$-non-sofic and is not minimal non-sofic.
\end{proposition}
\begin{proof}
Let
\[
 U_0=0\mathcal C\times\mathcal C^{n-1},\qquad
 U_1=1\mathcal C\times\mathcal C^{n-1},
\]
and let $c:\mathcal C^n\to U_0$ be the homeomorphism which prefixes
$0$ to the first coordinate. For $g\in\BTV{n}$, define
\[
 \iota(g)(x)=
 \begin{cases}
 cgc^{-1}(x),&x\in U_0,\\
 x,&x\in U_1.
 \end{cases}
\]
A finite prefix-replacement description of $g$ becomes one for
$\iota(g)$ by adding $0$ to the first-coordinate prefixes and adjoining
the identity piece on $U_1$. Thus $\iota(g)\in\BTV{n}$.
Restriction to $U_0$ shows both that $\iota$ is a homomorphism and that
it is injective. Every element of its image fixes $U_1$ pointwise.
The prefix replacement interchanging $U_0$ and $U_1$ belongs to
$\BTV{n}$ and does not have this property. Hence the image is proper.
The last assertion follows from the cited non-soficity theorem and
Lemma~\ref{lem:self-embedding}.
\end{proof}

The groups $\BTV{n}$ are simple by Brin \cite{Brin2010} and finitely
presented by Hennig--Matucci \cite[Theorem 25]{HennigMatucci}. Thus, among the known
simple non-sofic groups, both minimality and the existence of infinite
strictly descending chains of isomorphic non-sofic subgroups occur.
Simplicity by itself does not decide subgroup minimality.

\subsection{The Sandhu--Thom hyperbolic groups}
\label{sec:hyperbolic-examples}

Sandhu and Thom construct a torsion-free non-sofic hyperbolic group
$G_0$ as the fundamental group of a finite locally $\mathrm{CAT}(-1)$
triangle complex \cite[Theorem A]{SandhuThom}.
Their Corollary B then proves the existence of a non-trivial
torsion-free hyperbolic group $H$ with no non-trivial sofic quotient.
The latter is the group used in Section~\ref{sec:existence}.
The proof of that corollary is an existence argument; it does not
assert that $G_0$ itself has this stronger quotient property.

Since amenable groups are sofic, $H$ has no non-trivial amenable
quotient. In particular it is perfect and has no non-trivial finite
quotient. It therefore has no proper subgroup of finite index: the
coset action of such a subgroup would give a non-trivial finite
quotient. Thus the finite-index criterion does not apply to $H$.

The Tarski quotients in Theorem~\ref{thm:main-existence} give no
corresponding subgroup information. The arguments here do not decide
$\omega$-non-soficity for either $G_0$ or $H$.

\subsection{OpenAI's torsion-free hyperbolic group}
\label{sec:douglas-mghirbi}

Let $G$ be the torsion-free hyperbolic group studied by Douglas and
Mghirbi \cite{DouglasMghirbi}, originally constructed by OpenAI.
They prove that $G$ is non-sofic \cite[Theorem 1.1]{DouglasMghirbi}.
Their free-product decomposition
\cite[Proposition 6.4]{DouglasMghirbi} gives
\begin{equation}\label{eq:dm-splitting}
 G\cong G^\sharp*F_k,
 \qquad k=B+2r-3\geq1,
\end{equation}
where $G^\sharp$ is finitely presented, torsion-free and hyperbolic,
and $F_k$ is free of rank $k$. Here $B$ is the number of blocks and
$r$ is the integer in their construction. Their proof gives $B\geq2$,
so $k\geq1$ since $r\geq1$. We identify $G$ with the free product in
\eqref{eq:dm-splitting}.

\begin{proposition}\label{prop:dm-omega}
The free factor $G^\sharp$ is non-sofic. The group $G$ admits a proper
self-embedding and is therefore $\omega$-non-sofic and not minimal
non-sofic.
\end{proposition}
\begin{proof}
Free groups are sofic \cite[Section II.3]{CapraroLupini}. If
$G^\sharp$ were sofic, then \eqref{eq:dm-splitting} and the free-product
closure theorem \cite[Theorem 1(2)]{ElekSzabo2006} would make $G$
sofic, contrary to \cite[Theorem 1.1]{DouglasMghirbi}.

Choose a free basis $x_1,\ldots,x_k$ of $F_k$. The universal property
of the free product gives an endomorphism $\varphi:G\to G$ with
\[
 \varphi(a)=a\quad(a\in G^\sharp),
 \qquad
 \varphi(x_i)=x_i^2\quad(1\leq i\leq k).
\]
Write $F_k=\langle x_1\rangle*\cdots*\langle x_k\rangle$.
Replacing each non-zero exponent in a reduced word by twice that
exponent leaves a non-empty reduced word. Thus $\varphi|_{F_k}$ is
injective. The normal form in $G^\sharp*F_k$ then gives injectivity
of $\varphi$, since non-trivial syllables remain non-trivial;
see \cite[Section 2, Definition 2.1]{ElekSzabo2006}.

For properness, let $\chi:G\to\Z/2\Z$ be the homomorphism which
vanishes on $G^\sharp$, sends $x_1$ to $1$ and sends $x_i$ to $0$
for $i\geq2$. Then $\chi\circ\varphi=0$, whereas $\chi(x_1)=1$.
Consequently $x_1\notin\varphi(G)$.

It follows from Lemma~\ref{lem:self-embedding} that
\[
 G>\varphi(G)>\varphi^2(G)>\cdots
\]
is a strictly descending chain of groups isomorphic to $G$.
Explicitly,
$\varphi^j(G)=\langle G^\sharp,x_1^{2^j},\ldots,x_k^{2^j}\rangle$.
This proves $\omega$-non-soficity in the sense of
\cite[Definition 1.4]{Ersoy}.

\end{proof}

The retraction $G\twoheadrightarrow F_k$ gives a non-trivial sofic
quotient. Thus $G$ does not have the stronger quotient property of
$H$ in Theorem~\ref{thm:ST}. The proposition does not decide whether
$G^\sharp$ is minimal non-sofic or $\omega$-non-sofic.

\section{Further questions}
\label{sec:questions}

Theorem~\ref{thm:main-existence} supplies non-sofic Tarski monsters,
but it does not decide soficity for all Tarski monsters.

\begin{question}\label{q:sofic-monsters}
Does there exist a sofic Tarski monster of prime exponent? Does there
exist a sofic torsion-free Tarski monster?
\end{question}

\begin{question}\label{q:fp-minimal}
Does there exist a finitely presented minimal non-sofic group? Can
such a group be chosen torsion-free?
\end{question}

The quotient constructions used here do not establish finite
presentability. A quotient of a finitely presented group need not be
finitely presented.

\begin{question}[{\cite[Question 1.6]{Ersoy}}]\label{q:unbounded}
Does there exist a periodic minimal non-sofic group of unbounded
exponent?
\end{question}

Each $K_{p,m}$ has finite exponent. The group in
Corollary~\ref{cor:divisible} has unbounded exponent but contains
$T_p$ as a proper non-sofic subgroup.

\begin{question}\label{q:ff}
For the final quotients $Q$ in Fournier-Facio's construction, is
$\Gamma=\pi(P)$ non-sofic? Is every such $Q$ $\omega$-non-sofic?
\end{question}

An affirmative answer to the first question would answer the second
by Proposition~\ref{prop:ff-compression}. The non-soficity of the
universal group $P$ alone is not sufficient, since $\pi|_P$ need not
be injective.

\begin{question}\label{q:hyperbolic-omega}
Is every non-sofic hyperbolic group $\omega$-non-sofic?
\end{question}

Proposition~\ref{prop:dm-omega} gives an affirmative answer for the
group $G$ in Section~\ref{sec:douglas-mghirbi}, but not for its free
factor $G^\sharp$. The question also asks for more subgroup information
about the examples in \cite{SandhuThom}. Their non-sofic Tarski
quotients do not answer it.

\end{document}